\documentclass[11pt]{amsart} \textwidth=14.5cm \oddsidemargin=1cm
\usepackage[rgb]{xcolor}
\usepackage{tikz}
\usepackage{verbatim}
\usepackage{amsmath}
\usepackage{amsxtra}
\usepackage{amscd}
\usepackage{amsthm}
\usepackage{amsfonts}
\usepackage{amssymb}
\usepackage{eucal}
\usetikzlibrary{arrows}

\usepackage{latexsym,todonotes}

\usepackage[linktocpage=true]{hyperref}
\hypersetup{colorlinks,linkcolor=blue,urlcolor=blue,citecolor=blue}
\usepackage[all, cmtip]{xy}
\usepackage{stmaryrd}
\usepackage{color} 

\newtheorem{thm}{Theorem} 
\newtheorem{cor}[thm]{Corollary}
\newtheorem{lem}[thm]{Lemma}
\newtheorem{prop}[thm]{Proposition}

\theoremstyle{definition}
\newtheorem{definition}[thm]{Definition}

\theoremstyle{remark}
\newtheorem{rem}[thm]{Remark}

\numberwithin{equation}{section}
\begin{document}

\newcommand{\thmref}[1]{Theorem~\ref{#1}}
\newcommand{\secref}[1]{Section~\ref{#1}}
\newcommand{\lemref}[1]{Lemma~\ref{#1}}
\newcommand{\propref}[1]{Proposition~\ref{#1}}
\newcommand{\corref}[1]{Corollary~\ref{#1}}
\newcommand{\remref}[1]{Remark~\ref{#1}}
\newcommand{\eqnref}[1]{(\ref{#1})}

\newcommand{\exref}[1]{Example~\ref{#1}}

\newtheorem{innercustomthm}{{\bf Theorem}}
\newenvironment{customthm}[1]
  {\renewcommand\theinnercustomthm{#1}\innercustomthm}
  {\endinnercustomthm}
  
  \newtheorem{innercustomcor}{{\bf Corollary}}
\newenvironment{customcor}[1]
  {\renewcommand\theinnercustomcor{#1}\innercustomcor}
  {\endinnercustomthm}
  
  \newtheorem{innercustomprop}{{\bf Proposition}}
\newenvironment{customprop}[1]
  {\renewcommand\theinnercustomprop{#1}\innercustomprop}
  {\endinnercustomthm}

\newcommand{\Z}{{\mathbb Z}}
 \newcommand{\C}{{\mathbb C}}
 \newcommand{\N}{{\mathbb N}}
 \newcommand{\Q}{\mathbb{Q}}
 \newcommand{\la}{\lambda}
 \newcommand{\ep}{\epsilon}
 \newcommand{\bi}{\bibitem}
 \newcommand{\half}{\frac{1}{2}}
 \newcommand{\hgt}{\text{ht}}
 \newcommand{\mc}{\mathcal}
 \newcommand{\mf}{\mathfrak} 
 \newcommand{\hf}{\frac{1}{2}}
\newcommand{\ov}{\overline}
\newcommand{\ul}{\underline}
\newcommand{\I}{\mathbb{I}}
\newcommand{\id}{\text{id}}
\newcommand{\one}{\bold{1}}
\newcommand{\Qq}{\Q(q)}
\newcommand{\mA}{\mathcal{A}}
\newcommand{\tK}{\widetilde{K}}
\newcommand{\VV}{\mathbb{V}}
\newcommand{\WW}{\mathbb{W}}

\newcommand{\U}{\bold{U}}
\newcommand{\Udot}{\dot{\U}}
\newcommand{\B}{\bold{B}}
\newcommand{\vs}{\varsigma}

\newcommand{\ipsi}{\psi_{\imath}}
\newcommand{\Ui}{{\bold{U}^{\imath}}}
\newcommand{\UA}{{}_\mA{\bold{U}}}
\newcommand{\Uidot}{\dot{\bold{U}}^{\imath}}
\newcommand{\UAidot}{{}_\mA\dot{\bold{U}}^{\imath}}
\newcommand{\Iwhite}{\I_{\circ}}

\newcommand{\blue}[1]{{\color{blue}#1}}
\newcommand{\red}[1]{{\color{red}#1}}
\newcommand{\green}[1]{{\color{green}#1}}
\newcommand{\white}[1]{{\color{white}#1}}

\title[Canonical bases and Kazhdan-Lusztig theory]
{Canonical bases for tensor products and super Kazhdan-Lusztig theory}
 
 \author[Huanchen Bao]{Huanchen Bao}
\address{Department of Mathematics, National University of Singapore, Singapore 119076, Singapore.}
\email{huanchen@nus.edu.sg}

\author[Weiqiang Wang]{Weiqiang Wang}
\address{Department of Mathematics, University of Virginia, Charlottesville, VA 22904, United States.}
\email{ww9c@virginia.edu}

\author[Hideya Watanabe]{Hideya Watanabe}
\address{Research Institute for Mathematical Sciences, Kyoto University, Kyoto 606-8502, Japan}
\email{hideya@kurims.kyoto-u.ac.jp}

\begin{abstract}
We generalize a construction in \cite{BW18b} by showing that, for a quantum symmetric pair $(\bf U, \bf U^\imath)$ of finite type, the tensor product of a based $\bf U^\imath$-module and a based $\bf U$-module is a based $\bf U^\imath$-module. This is then used to formulate a  Kazhdan-Lusztig theory for an {\em arbitrary} parabolic BGG category $\mathcal{O}$ of the ortho-symplectic Lie superalgebras, extending a main result in \cite{BW18a}.
\end{abstract}

\maketitle

\let\thefootnote\relax


\section{Introduction}

\subsection{}

Following Lusztig \cite{Lu94}, we shall refer to $(M, \B)$, which consists of a $\U$-module $M$ and its canonical basis $\B$, as a based $\U$-module, where $\U$ is a Drinfeld-Jimbo quantum group of finite type. Examples of such based $\U$-modules include any finite-dimensional simple $\U$-module or a tensor product of several such simple $\U$-modules. The canonical basis on a tensor product of several finite-dimensional simple $\U$-modules was constructed by Lusztig \cite{Lu92}, and it has found applications to the Kazhdan-Lusztig theory for general linear Lie superalgebra $\mathfrak{gl}(m|n)$ of type A \cite{Br03, CLW15}. 

\subsection{}

As a generalization of canonical bases for quantum groups, a theory of canonical basis arising from quantum symmetric pairs (QSP, for short) $(\U, \Ui)$ of finite type is systematically developed in \cite{BW18b}. For any finite-dimensional based $\U$-module $(M, \B)$, a new bar involution $\ipsi$ on $M$ was formulated and a $\ipsi$-invariant basis $\B^\imath$ of $M$ (called an $\imath$-canonical basis for $M$) was constructed (see \cite[Theorem 5.7]{BW18b}), which satisfies some specific properties when expanded with respect to $\B$; we shall call $(M, \B^\imath)$ a based $\Ui$-module. 

The first examples of based $\Ui$-modules were constructed in \cite{BW18a} for quasi-split QSP $(\U, \Ui)$ of type AIII. The $\imath$canonical basis on a tensor product of $\U$-modules was used to formulate a Kazhdan-Lusztig theory on the full BGG category for ortho-symplectic Lie superalgebra  of type B \cite{BW18a} (for type D see \cite{Bao17}); also see \cite{CLW11} on some parabolic BGG category via a super duality approach.

\subsection{}

This paper is intended to supplement the two earlier papers \cite{BW18a, BW18b} of the first two authors on canonical bases arising from quantum symmetric pairs and applications to super Kazhdan-Lusztig theory; it extends two principal results on $\imath$-canonical basis and super Kazhdan-Lusztig theory therein to full generalities.

\subsection{}

By definition of a QSP $(\U, \Ui)$, $\Ui$ is a coideal subalgebra of $\U$ \cite{Le99}; that is, the comultiplication $\Delta$ on $\U$ when restricting to $\Ui$ satisfies $\Delta: \Ui \rightarrow \Ui \otimes \U$. Hence $M\otimes N$ is a $\Ui$-module for any $\Ui$-module $M$ and $\U$-module $N$. In the first main theorem (see Theorem \ref{thm:1}) we show that, for a based $\Ui$-module $(M, \B^\imath)$ and a based $\U$-module $(N, \B)$, there exists a $\ipsi$-invariant basis $\B^\imath \diamondsuit_{\imath} \B$ on the $\Ui$-module $M\otimes N$ such that $(M\otimes N, \B^\imath \diamondsuit_{\imath} \B)$ is a based $\Ui$-module. This generalizes a main result in \cite{BW18b} on the $\imath$-canonical basis on a tensor product of $\U$-modules, since a based $\Ui$-module which is not a $\U$-module exists (cf. \cite{BW18a}). 

The construction of the new bar involution $\ipsi$ on $M\otimes N$ above uses a certain element $\Theta^\imath$ in a completion of $\Ui\otimes \U^+$, which was due to \cite{BW18a} for quasi-split QSP of type AIII/IV and then established in Kolb \cite{Ko17} in full generality with an elegant new proof. We establish the integrality of $\Theta^\imath$ by using the integrality of the quasi-$\mathcal R$ matrix in \cite{Lu94} and the integrality of the quasi-$\mc K$ matrix in \cite{BW18b}. To construct the $\imath$-canonical basis on $M\otimes N$, we use crucially a partial order, which is different from and simpler than the old one used in \cite{BW18b} even when $M$ is a $\U$-module; the old partial order does not make much sense in our new setting. 

\subsection{}

For the quasi-split QSP $(\U, \Ui)$ of type AIII/AIV, the quantum group $\U$ is of type A; we let $\VV$ and $\WW$ denote the natural representation of $\U$ and its dual. In this case, the $\imath$-canonical basis on a based $\U$-module was first constructed in \cite{BW18a} when a certain parameter $\kappa=1$ (also see \cite{Bao17} with parameter $\kappa=0$). The super Kazhdan-Lusztig theory for the {\em full} BGG category $\mc O_{\bf b}$ of an ortho-symplectic Lie superalgebra $\mathfrak g$ of type B in \cite{BW18a} (for type D see \cite{Bao17}) of integer or half-integer weights was formulated via the $\imath$-canonical basis on a mixed tensor $\U$-module with $m$ copies of $\VV$ and $n$ copies of $\WW$, where the order of the tensor product depends on the choice of a Borel subalgebra $\bf b$ in $\mathfrak g$. 

As a consequence, a Kazhdan-Lusztig theory for parabolic categories $\mc O^{\mathfrak l}_{\bf b}$, where the Levi subalgebra $\mathfrak l$ in $\mathfrak g$ is a product of Lie algebras of type A, can be formulated and established via the $\imath$-canonical basis on a tensor product $\U$-module $\mathbb T$ of various exterior powers of $\VV$ and of $\WW$. Note however not all the parabolic categories of $\mathfrak g$-modules arise in this way; indeed a general Levi sugalgebra of $\mathfrak g$ is isomorphic to a product of several Lie subalgebras of type A and a Lie subalgebra of type B. 


Theorem \ref{thm:1}, when specialized for the QSP $(\U, \Ui)$ of quasi-split type AIII/AIV, provides an $\imath$-canonical basis for a $\Ui$-module on the tensor product of the form $\wedge^{a} \mathbb{V}_- \otimes \mathbb T$. Here $\mathbb T$ is a tensor product $\U$-module of various exterior powers of $\VV$ and of $\WW$, while $\wedge^{a} \mathbb{V}_-$ (for $a>0$) is a ``type $B$" exterior power, which is a $\Ui$-module but not a $\U$-module. These new $\imath$-canonical bases are used to formulate the super Kazhdan-Lusztig theory for an {\em arbitrary}  parabolic BGG category $\mc O$ of the ortho-symplectic Lie superalgebras of type B and D; see Theorem~\ref{thm:2}. The super Kazhdan-Lusztig polynomials $t^{\bf b}_{gf}(q)$ admit a positivity property; see Theorem~\ref{thm:positivity}. 

\subsection{}

This paper is organized as follows. Theorem~\ref{thm:1} and its proof are presented in Section~\ref{sec:based}, and we shall follow notations in \cite{BW18b} throughout Section~\ref{sec:based}. The formulation of Theorem~\ref{thm:2} is given in Section~\ref{sec:KL}; its proof basically follows the proof for the Kazhdan-Lusztig theory for the full category $\mc O$ in \cite[Part 2]{BW18a} once we have Theorem~\ref{thm:1} available to us. We shall follow notations in \cite{BW18a} throughout Section~\ref{sec:KL}. To avoid much repetition, we refer precisely and freely to the two earlier papers \cite{BW18a, BW18b}.

\vspace{.3cm}

{\bf Acknowledgements.}
The research of WW  is partially supported by an NSF grant DMS-1702254.

\section{Tensor product modules as based $\Ui$-modules}
 \label{sec:based}
 
\subsection{}

We shall follow the notations in \cite{BW18b} throughout this section.  

Let $\U$ denote a quantum group of finite type over the field $\Q(q)$ associated to a root datum of type $(\I, \cdot)$, and let $\Delta$ denote its comultiplication as in \cite{Lu94}. We denote the bar involution on $\U$ or its based module by $\psi$. 

Let $\Ui \subset \U$ be a coideal subalgebra associated to a Satake diagram such that $(\U,\Ui)$ forms a quantum symmetric pair \cite{Le99}. Let $\mathcal{A} := \Z[q,q^{-1}]$. Let $\Uidot$ be the modified version of $\Ui$  and let $\UAidot$ be its $\mathcal{A}$-form, respectively; see \cite[\S3.7]{BW18b}. Let $\ipsi$ be the bar involution on $\Ui$, $\Uidot$ and $\UAidot$. Let $X_\imath$ be the $\imath$-weight lattice  \cite[(3.3)]{BW18b}. A weight (i.e., $X_\imath$-weight) module of $\Ui$ can be naturally regarded as a $\Uidot$-module. 

We introduce based $\Ui$-modules generalizing \cite[\S27.1.2]{Lu94}. Let $\mathbf{A} = \Q[[q^{-1}]] \cap \Qq$. We write $-\otimes- = - \otimes_{\Qq} -$ whenever the base ring is  $\Qq$.

\begin{definition}\label{ad:def:1}
Let $M$ be a finite-dimensional weight $\Ui$-module over $\Qq$ with a given $\Qq$-basis $\B^\imath$. The pair $(M, \B^\imath)$ is called a based $\Ui$-module if the following conditions are satisified:
	\begin{enumerate}
		\item 	$\B^\imath \cap M_{\nu}$ is a basis of $M_{\nu}$, for  any $\nu \in X_\imath$;
		\item 	The $\mA$-submodule ${}_{\mA}M$ generated by $\B^\imath$ is stable under ${}_{\mA}\Uidot$;
		\item 	The $\Q$-linear involution $\ipsi: M \rightarrow M$ defined by 
		$\ipsi(q)= q^{-1}, \ipsi(b) = b$ for all $b \in \B^\imath$ 
		is compatible with the $\Uidot$-action, i.e., $\ipsi(um) = \ipsi(u) \ipsi(m)$, for all $u\in \Uidot, m\in M$;
		\item 	Let $L(M)$ be the $\mathbf{A}$-submodule of $M$ generated by $\B^\imath$. Then the image of $\B^\imath$ in $L(M)/ q^{-1}L(M)$ forms a $\Q$-basis in $L(M)/ q^{-1}L(M)$. 
	\end{enumerate}
\end{definition}	
We shall denote by $\mathcal L(M)$ the $\Z[q^{-1}]$-span of $\B^\imath$; then $\B^\imath$ forms a $\Z[q^{-1}]$-basis for $\mathcal L(M)$. (There are similar constructions for a based $\U$-module in similar notations.)

\subsection{}
\label{subsec:Up}

Let $\Upsilon =\sum_\mu \Upsilon_\mu$ (with $\Upsilon_0=1$ and $\Upsilon_\mu \in \U^+_\mu$) be the intertwiner (also called quasi-$\mathcal K$ matrix) of the quantum symmetric pair $(\U, \Ui)$ introduced in \cite[Theorem~ 2.10]{BW18a}; for full generality see \cite[Theorem 6.10]{BK19}, \cite[Theorem 4.8, Remark~ 4.9]{BW18b}). 
It follows from \cite[Theorem~5.7]{BW18b} (also cf. \cite[Theorem~ 4.25]{BW18a}) that a based $\U$-module $(M,\B)$ becomes a based $\Ui$-module with a new basis $\B^\imath$ (which is uni-triangular relative to $\B$) with respect to the involution $\ipsi := \Upsilon \circ \psi$.

Let $\widehat{\U \otimes \U}$ be the completion of the $\Qq$-vector space $\U \otimes \U$ 
with respect to the descending sequence of subspaces 
\[
\U \otimes \U^- \U^0 \big(\sum_{\hgt(\mu) \geq N}\U_{\mu}^+ \big) + \U^+ \U^0 \big(\sum_{\hgt(\mu) \geq N}\U_{\mu}^- \big) \otimes \U,   \text{ for }N \ge 1, \mu \in \Z\I.
\]
We have the obvious embedding of $\U \otimes \U$ into $\widehat{\U \otimes \U}$. 
By continuity the $\Q(q)$-algebra structure on $\U \otimes \U$ extends to a $\Q(q)$-algebra structure on $ \widehat{\U \otimes \U}$. We know the quasi-$\mc R$ matrix $\Theta$ lies in $\widehat{\U \otimes \U}$ by \cite[Theorem~4.1.2]{Lu94}. It follows from \cite[Theorem~4.8]{BW18b} and \cite[Theorem~6.10]{BK19} that $\Upsilon^{-1} \otimes \id$ and $\Delta(\Upsilon)$ are both in $\widehat{\U \otimes \U}$.

We define 
\begin{equation}\label{ad:eq:1}
\Theta^\imath = \Delta (\Upsilon) \cdot \Theta \cdot (\Upsilon^{-1} \otimes \id) \in \widehat{\U \otimes \U}.
\end{equation}
We can write 
\begin{equation}
  \label{eq:Thetamu}
\Theta^\imath = \sum_{\mu \in \N\I}\Theta^\imath_{\mu}, \qquad
\text{ where } \Theta^\imath_{\mu} \in  \U \otimes \U^+_\mu.
\end{equation}
The following result first appeared in \cite[Proposition~3.5]{BW18a} for the quantum symmetric pairs of (quasi-split) type AIII/AIV.
\begin{lem}\cite[Proposition~3.10]{Ko17}\label{ad:lem:1}
We have $\Theta^\imath_{\mu} \in \Ui \otimes \U^+_\mu$, for all $\mu$. (The element $\Theta^\imath_{\mu}$ is denoted by $R^\theta_\mu$ in \cite{Ko17}.)
\end{lem}

Another basic ingredient which we shall need is the integrality property of $\Theta^\imath$. 
\begin{lem}
    \label{lem:Z}
We have $\Theta^\imath_{\mu} \in \UA \otimes_{\mA} \UA^+_\mu$, for all $\mu$.
\end{lem}

\begin{proof}
By a result of Lusztig \cite[24.1.6]{Lu94}, we have $\Theta =\sum_{\nu \in\N\I} \Theta_\nu$ is integral, i.e., $\Theta_\nu \in \UA^-_\nu  \otimes_{\mA} \UA^+_\nu$. By \cite[Theorem~5.3]{BW18b} we have $\Upsilon=\sum_{\mu \in\N\I} \Upsilon_\mu$ is integral, i.e., $\Upsilon_\mu \in \UA^+_\mu$ for each $\mu$; it follows that $\Upsilon^{-1} = \psi(\Upsilon)$ is integral too thanks to \cite[Corollary~4.11]{BW18b}. It is well known that the comultiplication $\Delta$ preserves the $\mA$-form, i.e., $\Delta (\UA) \subset \UA  \otimes_{\mA} \UA$. The lemma follows now by the definition of $\Theta^\imath$ in \eqref{ad:eq:1}. 
\end{proof}

\subsection{}

Define a partial order $<$ on $X$ by setting $\mu'<\mu$ if $\mu' - \mu \in \N \I$. Denote by $|b | =\mu$ if an element $b$ in a $\U$-module is of weight $\mu$. 
Now we are ready to prove the first main result of this paper. 
\begin{thm}\label{thm:1}
Let $(M, \B^\imath)$ be a based $\Ui$-module and $(N, \B)$ be a based $\U$-module.  
\begin{enumerate}
\item
For $b_1 \in \B^\imath, b_2\in \B$, there exists a unique element $b_1\diamondsuit_\imath b_2$ which is $\ipsi$-invariant such that $b_1\diamondsuit_\imath b_2\in b_1\otimes b_2 +q^{-1}\Z[q^{-1}] \B^\imath \otimes \B$.
\item
We have $b_1\diamondsuit_\imath b_2 \in b_1\otimes b_2 +\sum\limits_{(b'_1,b'_2) \in \B^\imath \times \B, |b_2'| < |b_2|} q^{-1}\Z[q^{-1}] \, b_1' \otimes b_2'$.
\item
$\B^\imath \diamondsuit_\imath \B :=\{b_1\diamondsuit_\imath b_2 \mid b_1 \in \B^\imath, b_2\in \B \}$ forms a $\Qq$-basis for  $M \otimes N$, an $\mA$-basis for ${}_\mA M  \otimes_{\mA} {}_\mA N$, and a $\Z[q^{-1}]$-basis for $\mathcal L(M)  \otimes_{\Z[q^{-1}]} \mathcal L(N)$. (This is called {\em the $\imath$-canonical basis} for $M\otimes N$.)
\item
$(M\otimes N, \B^\imath \diamondsuit_\imath \B)$ is a based $\Ui$-module.
\end{enumerate}
\end{thm}

\begin{proof}
It follows by Lemma~\ref{ad:lem:1} that the element $\Theta^\imath$ gives rise to a well-defined operator on the tensor product $M \otimes N$. Following \cite[(3.17)]{BW18a}, we define a new bar involution on $M\otimes N$ (still denoted by $\ipsi$) by letting 
\[
\ipsi := \Theta^\imath \circ (\ipsi \otimes \psi) : M \otimes N \longrightarrow M \otimes N.
\]
Recall from \cite{Lu94} that $\Delta(E_i)  = E_i \otimes 1 + \widetilde{K}_i \otimes E_i$. If follows that 
\[
\Delta (\Upsilon) \in \Upsilon \otimes 1 + \sum_{0\neq \mu\in \N\I} \U\otimes \U^+_\mu.
\]
Recalling \eqref{eq:Thetamu}, we have
\begin{equation}\label{ad:eq:theta}
\Theta^\imath_0 = (\Upsilon \otimes 1) \cdot (1 \otimes 1) \cdot  (\Upsilon^{-1} \otimes 1)=1 \otimes 1.
\end{equation}

Let $b_1 \in \B^\imath$ and  $b_2 \in \B$. By \eqref{ad:eq:theta} and Lemma~\ref{lem:Z}, we have  
\begin{equation}
  \label{ad:eq:order}
\ipsi (b_1 \otimes b_2) \in  b_1 \otimes b_2 + \sum_{\substack{(b'_1,b'_2) \in \B^\imath \times \B \\ |b_2'| < |b_2|}} \mA \, b'_1 \otimes b'_2. 
\end{equation}
Applying \cite[Lemma~24.2.1]{Lu94}, there exists a $\ipsi$-invariant element $(b_1 \otimes b_2)^\imath \in M \otimes N$ such that 
\[
	b_1 \diamondsuit_\imath b_2  \in b_1 \otimes b_2 + \sum_{\substack{(b'_1,b'_2) \in \B^\imath \times \B \\ |b_2'| < |b_2|}} q^{-1} \Z[q^{-1}]\, b'_1 \otimes b'_2.  
\]
This proves (2), and Part (3) follows immediately. 

A by now standard argument shows the uniqueness of $b_1 \diamondsuit_\imath b_2$ as stated in (1); note a weaker condition than (2) is used in (1).

It remains to see that $(M, \B^\imath \diamondsuit_\imath \B)$ is a based $\Ui$-module. The item (3) in the definition of a based $\Ui$-module is proved in the same way as for \cite[Proposition~ 3.13]{BW18a}, while the remaining items are clear. 

This completes the proof. 
\end{proof}

\begin{rem}  
\begin{enumerate}
\item
An elementary but key new ingredient in Theorem~\ref{thm:1} above is the use of a (coarser) partial order $<$, which is different from the partial order $<_{\imath}$ used in \cite[(5.2)]{BW18b}. 
\item
Assume that $M$ is a based $\U$-module. Then the $\imath$-canonical basis for $M\otimes N$  in Theorem~\ref{thm:1} coincides with the one in \cite[Theorem 5.7]{BW18b},  thanks to the uniqueness in Theorem~\ref{thm:1}(1). 
\end{enumerate}
\end{rem}

\begin{rem}
 \label{rem:general}
Theorem~\ref{thm:1} would be valid whenever we can establish the (weaker) integrality of $\Theta^\imath$ acting on $M\otimes N$. This might occur when we consider more general parameters for $\Ui$ than \cite{BW18b} or when we consider quantum symmetric pairs of Kac-Moody type in a forthcoming work of the first two authors. 
\end{rem}


\section{Applications to super Kazhdan-Lusztig theory}
 \label{sec:KL}

\subsection{}
In this section, we shall apply Theorem~\ref{thm:1} to formulate and establish the (super) Kazhdan-Lusztig theory for an arbitrary parabolic category $\mc{O}$ of modules of integer or half-integer weights for ortho-symplectic Lie superalgebras, generalizing \cite[Part ~2]{BW18a} (also see \cite{Bao17}). We shall present only the details on an arbitrary parabolic category $\mc{O}$ consisting of modules of integer weights for Lie superalgebra $\mathfrak{osp}(2m+1\vert 2n)$.

\subsection{}

All relevant notations throughout this section shall be consistent with \cite[Part 2]{BW18a}. In particular, we use a comultiplication for $\U$ different from \cite{Lu94}; this leads to a version of the intertwiner $\Upsilon =\sum_\mu \Upsilon_\mu$ with $\Upsilon_\mu \in \widehat{\U}^-$ (compare with \S\ref{subsec:Up}), and a version of Theorem~\ref{thm:1} in which the opposite partial order and the lattice $\Z[q]$ 
 are used. To further match notations with \cite{BW18a} in this section, we denote the $\mA$-form of any based $\U$ or $\Ui$-modules $M$, as $M_{\mA}$ (instead of ${}_\mA M$).

We consider the infinite-rank quantum symmetric pair $(\U, \Ui)$ as defined in \cite[Section~8]{BW18a} (where the parameter is chosen to be $\kappa=1$ in the notation of \cite{BW18b}). It is a direct limit of quantum symmetric pairs of type AIII, $(\U(\mathfrak{sl}_N), \Ui(\mathfrak{sl}_N) )$, for $N$ even. 
We denote by $\mathbb{V}$ the natural representation of $\U$, and by $\mathbb{W}$ the restricted dual of $\mathbb{V}$. 

Associated to any given $0^m1^n$-sequence ${\bf b} = (b_1,\ldots,b_{m+n})$ starting with $0$, we have a fundamental system of $\mathfrak{osp}(2m+1\vert 2n)$, denoted by $\Pi_{\bf{b}} = \{ -\epsilon_1^{b_1}, \epsilon_i^{b_i} - \epsilon_{i+1}^{b_{i+1}} \mid 1 \leq i \leq m+n-1 \}$; here $\epsilon_i^0 = \epsilon_x$ for some $1 \leq x \leq m$ and $\epsilon_j^1 = \epsilon_{\bar{y}}$ for some $1 \leq y \leq n$ so that $\{\epsilon_i^{b_i} \mid 1\le i \le m+n\}$ form a permutation of $\{\epsilon_a, \epsilon_{b} \mid 1\le a \le m, 1\le b \le n\}$.


\subsection{}

Let $W_{B_{s}}$ and $W_{A_{s-1}}$ be the Weyl group of type $B_{s}$ and type $A_{s-1}$ with unit $e$, respectively. We denote their corresponding Hecke algebras by $\mc{H}_{B_{s}}$ and $\mc{H}_{A_{s-1}}$, with Kazhdan-Lusztig bases by $\{\underline{H_{w}} \vert w \in W_{B_s}\}$ and $\{\underline{H_{w}} \vert w \in W_{A_{s-1}}\}$, respectively. Both algebras act naturally on the right on $\mathbb{V}^{\otimes s}$ and $\mathbb{W}^{\otimes s}$; cf. \cite[Section 5]{BW18a}. We define
\begin{align*}
\wedge^{s} \mathbb{V}_- &= \mathbb{V}^{\otimes s} \Big / \sum_{e \neq w \in W_{B_{s}}}\mathbb{V}^{\otimes s} \cdot \underline{H_w},\\
 \wedge^{s} \mathbb{V}&=  \mathbb{V}^{\otimes s} \Big / \sum_{e \neq w \in W_{A_{s-1}}} \mathbb{V}^{\otimes s} \cdot \underline{H_w}.
\end{align*}

We similarly define $\wedge^{s} \mathbb{W}_-$ and $\wedge^{s} \mathbb{W}$. Note $\wedge^{s} \mathbb{V}_-, \wedge^{s} \mathbb{V}, \wedge^{s} \mathbb{W}_-$ and $\wedge^{s} \mathbb{W}$ are all based $\Ui$-modules by  \cite[Theorem~5.8]{BW18a}. We shall denote
$$
\mathbb{V}^{c} := \begin{cases}
\mathbb{V} \quad & \text{ if } c = 0, \\
\mathbb{W} \quad & \text{ if } c = 1.
\end{cases}
$$

The following corollary is a direct consequence of Theorem~\ref{thm:1}. 

\begin{cor}\label{cor:1}
Let $c_1,\ldots,c_k \in \{ 0,1 \}$ and $a_0, a_1,\ldots,a_k \in \N$. Then (a suitable completion of) the tensor product    
\[
\mathbb{T}^{\bf{b}, \mathfrak{l}} = \wedge^{a_0} \mathbb{V}_- \otimes \wedge^{a_1} \mathbb{V}^{c_1} \otimes \cdots \otimes \wedge^{a_k} \mathbb{V}^{c_k}
\]
is a based $\Ui$-module. 
\end{cor}
The completion above arises since we deal with quantum symmetric pairs of infinite rank, and it is a straightforward generalization of the $B$-completion studied in \cite[Section~9]{BW18a}. Note that the $\imath$-canonical basis lives in $\mathbb{T}^{\bf{b}, \mathfrak{l}}$ (instead of its completion) by Theorem~\ref{thm:positivity} below.

\subsection{}

Associated to the fundamental system $\Pi_{\bf{b}}$ are the set of positive roots $\Phi^+_{\mathbf{b}}$ and the Borel subalgebra $\mathfrak b_{\bf b}$ of $\mathfrak{osp}(2m+1 \vert 2n)$. 
Let $\Pi_{\mathfrak{l}} \subset \Pi_{\bf{b}}$ be a subset of even simple roots. We introduce the corresponding Levi subalgebra $\mathfrak l$ and parabolic subalgebra $\mathfrak p$ of $\mathfrak{osp}(2m+1 \vert 2n)$:
\begin{align*}
\mathfrak{l} = \mathfrak h_{m|n} \bigoplus \bigoplus_{\alpha \in \mathbb{Z}\Pi_{\mathfrak{l}} \cap \Phi_{\mathbf{b}}} \mathfrak{osp}(2m+1\vert 2n)_{\alpha},
\qquad\quad 
\mathfrak{p} = \mathfrak{l} + \mathfrak b_{\bf b}.
\end{align*}
Recall \cite[\S7]{BW18a} the weight lattice $X(m|n) = \sum^{m}_{i=1}\Z\epsilon_{i} + \sum^n_{j=1}\Z\epsilon_{\ov j}.$
We denote
\[
X_{\mathbf{b}}^{\mathfrak l,+} = \{\lambda \in X(m|n) \mid (\lambda \vert \alpha) \ge 0, \forall \alpha \in  \Pi_{\mathfrak{l}}\}.
\] 
Let $L_0(\lambda)$ be the irreducible $\mathfrak{l}$-module with highest weight $\lambda$, which is extended trivially to a $\mathfrak{p}$-module. We form the parabolic Verma module
\[
M^{\mathfrak{l}}_{{\bf b}}(\lambda) :=\text{Ind}^{\mathfrak{osp}(2m+1|2n)}_{\mathfrak{p}}L_0(\lambda).
\]

\begin{definition}
Let $\mc{O}^{\mathfrak{l}}_{\mathbf{b}}$ be the category of $\mathfrak{osp}(2m+1\vert 2n)$-modules $M$ such that 
\begin{itemize}
\item[(i)]
$M$ admits a weight space decomposition $M=\bigoplus\limits_{\mu \in X(m|n)}M_\mu$, and $\dim M_\mu<\infty$;

\item[(ii)]
M decomposes over $\mathfrak{l}$ into a direct sum of $L_\mathfrak{l}(\lambda)$ for some $\lambda \in X_{\mathbf{b}}^{\mathfrak l,+}$;
\item[(iii)]
there exist finitely many weights ${}^1\la,{}^2\la,\ldots,{}^k\la\in X_{\mathbf{b}}^{\mathfrak l,+}$
(depending on $M$) such that if $\mu$ is a weight in $M$, then
$\mu\in{{}^i\la}-\sum_{\alpha\in{\Pi_{{\bf b}}}}\N\alpha$, for
some $i$.
\end{itemize}
The morphisms in $\mc{O}^{\mathfrak{l}}_{\mathbf{b}}$ are all (not necessarily even)
homomorphisms of $\mathfrak{osp}(2m+1|2n)$-modules.
\end{definition}

For $\lambda \in X_{\mathbf{b}}^{\mathfrak l,+}$, we shall denote by $L^{\mathfrak{l}}_{{\bf b}}(\lambda)$ the simple quotient of the parabolic Verma module $M^{\mathfrak{l}}_{{\bf b}}(\lambda)$ in $\mc{O}^{\mathfrak{l}}_{\mathbf{b}}$ with highest weight $\lambda$. Following \cite[Definition~7.4]{BW18a}, we can define the tilting modules $T^{\mathfrak{l}}_{{\bf b}}(\lambda)$ in $\mc{O}^{\mathfrak{l}}_{\mathbf{b}}$, for $\lambda \in X_{\mathbf{b}}^{\mathfrak l,+}$. We denote by $\mc{O}^{\mathfrak{l}, \Delta}_{\mathbf{b}}$ the full subcategory of $\mc{O}^{\mathfrak{l}}_{\mathbf{b}}$ generated by all modules possessing finite parabolic Verma flags. 

\subsection{}

Recall the bijection $X(m \vert n) \leftrightarrow I^{m+n}$ \cite[\S8.4]{BW18a}, where an element $f \in I^{m+n}$ is understood as a $\rho$-shifted weight.  We consider the restriction $X_{\mathbf{b}}^{\mathfrak l,+} \leftrightarrow I_{\mathfrak{l}, +}^{m+n}$, where the index set $I_{\mathfrak{l}, +}^{m+n}$ is defined as the image under the bijection. 

Let $W_{\mathfrak{l}}$ be the Weyl group of $\mathfrak{l}$ with the corresponding Hecke algebra $\mc{H}_{\mathfrak{l}}$. Recall that $\Pi_{\mathfrak{l}} \subset \Pi_{\bf{b}}$ is a subset of even simple roots. Hence
we have the natural right action of $\mc{H}_\mathfrak{l}$ on the $\mA$-module $\mathbb{T}_{\mA}^{\bf{b}} := \mathbb{V}_{\mathcal{A}}^{b_1} \otimes_{\mA} \cdots \otimes_{\mA}  \mathbb{V}_{\mathcal{A}}^{b_{m+n}}$ with a standard basis $M^{\bf{b}}_{f} \in \mathbb{T}^{\bf{b}}_{\mA}$, for $f \in I_{\mathfrak{l}}^{m+n}$;cf. \cite[\S8.2]{BW18a}.
We define 
\[
 \mathbb{T}^{\bf{b}, \mathfrak{l}}_{\mA} = \mathbb{T}^{\bf{b}}_{\mA} \Big / \sum_{e \neq w \in W_{\mathfrak{l}}} \mathbb{T}_{\mA}^{\bf{b}} \cdot \underline{H_w}.
\]
The quotient space is an $\mA$-form $\mathbb{T}^{\bf{b}, \mathfrak{l}}_{\mA}$ of the $\Qq$-space $\mathbb{T}^{\bf{b}, \mathfrak{l}}$ appearing in Corollary~\ref{cor:1}:
\begin{equation}
  \label{eq:T-}
 \mathbb{T}^{\bf{b}, \mathfrak{l}}_{\mA} =
 \wedge^{a_0} \mathbb{V}_{-,\mA} \otimes_{\mA}  \wedge^{a_1} \mathbb{V}_\mA^{c_1} \otimes_{\mA}  \cdots \otimes_{\mA}  \wedge^{a_k} \mathbb{V}_\mA^{c_k}, \qquad \text{ for } c_i \in \{0, 1\}, \quad a_i \in \mathbb{N},
\end{equation}
where $c_i$ and $a_i$ are determined as follows. Let $W'$ denote a subgroup of the Weyl group of $\mathfrak{osp}(2m+1 \vert 2n)$,  $W' = W_{B_{m}} \times S_{n} = \langle s_0,s_1,\ldots,s_{m-1}, s_{m+1}, \ldots, s_{m+n-1} \rangle$, where $s_i =s_{\alpha_i}$, and $\alpha_0 = -\epsilon_1^{b_0}$, $\alpha_i = \epsilon_{i}^{b_i} - \epsilon_{i+1}^{b_{i+1}}$ for $1 \leq i \leq m+n-1$. Then, $W_{\mathfrak{l}}$ is the parabolic subgroup of $W'$ generated by $\{ s_i \mid \alpha_i \in \Pi_{\mathfrak{l}} \}$. Let us write $\{ 0,1,\ldots,m+n \} \setminus \{ i \mid \alpha_i \in \Pi_{\mathfrak{l}} \} = \{ j_1 < j_2 <\cdots < j_{k+1} \}$. Then, $a_i = j_{i+1} - j_{i}$ and $c_{i+1} = b_{j_i}$, where it is understood that $j_0 = 0$.  

For any standard basis element $M^{\bf{b}}_{f} \in \mathbb{T}^{\bf{b}}_{\mA}$ with $f \in I_{\mathfrak{l}, +}^{m+n}$, we denote by $M^{\bf{b}, \mathfrak{l}}_{f}$ its image in $\mathbb{T}^{\bf{b}, \mathfrak{l}}_{\mA}$. Then $\{M^{\bf{b}, \mathfrak{l}}_{f} \vert f \in I_{\mathfrak{l}, +}^{m+n}  \}$ forms an $\mA$-basis of  $\mathbb{T}^{\bf{b}, \mathfrak{l}}_{\mA}$. Let 
\[
\mathbb{T}^{\bf{b}, \mathfrak{l}}_{\Z}  =  \mathbb{T}^{\bf{b}, \mathfrak{l}}_{\mA} \otimes_{\mA} \Z 
\]
be the specialization of $\mathbb{T}^{\bf{b}, \mathfrak{l}}_{\mA}$ at $q=1$.
Let $\widehat{\mathbb{T}}^{\bf{b}, \mathfrak{l}}_{\Z}$ be the $B$-completion of $\mathbb{T}^{\bf{b}, \mathfrak{l}}_{\Z}$ following \cite[Section~9]{BW18a}.
It follows from Corollary~\ref{cor:1} the space $\widehat{\mathbb{T}}^{\bf{b}, \mathfrak{l}}_{\Z}$ admits the $\imath$-canonical basis $\{T^{\bf{b}, \mathfrak{l}}_f\vert f \in I_{\mathfrak{l}, +}^{m+n}\}$. We can similarly define the dual $\imath$-canonical basis $\{L^{\bf{b}, \mathfrak{l}}_f \vert f \in I_{\mathfrak{l}, +}^{m+n}\}$ of $\widehat{\mathbb{T}}^{\bf{b}, \mathfrak{l}}_{\Z}$ following \cite[Theorem~9.9]{BW18a}.

\subsection{}

We denote by $[\mc{O}^{\mathfrak{l}, \Delta}_{\mathbf{b}}]$ the Grothendieck group of the category $\mc{O}^{\mathfrak{l}, \Delta}_{\mathbf{b}}$. We have the following isomorphism of $\Z$-modules:
\begin{align*}
\Psi:[\mc{O}^{\mathfrak{l}, \Delta}_{\mathbf{b}}]&\longrightarrow   \mathbb{T}^{\bf{b}, \mathfrak{l}}_{\Z}\\
M^{\mathfrak{l}}_{\bf{b}}(\lambda)&\mapsto M^{{\bf b},{\mathfrak{l}}}_{f^{{\bf b}}_{\lambda}}(1), \quad \quad \text{ for } \lambda \in X_{\mathbf{b}}^{\mathfrak l,+}.
\end{align*}
We define $[[\mc{O}^{\mathfrak{l}, \Delta}_{\mathbf{b}}]]$ as the completion of $[\mc{O}^{\mathfrak{l}, \Delta}_{\mathbf{b}}]$ such that the extension of $\Psi$, 
\[
 \Psi: [[\mc{O}^{\mathfrak{l}, \Delta}_{\mathbf{b}}]] \longrightarrow \widehat{\mathbb{T}}_{\Z}^{\bf{b}, \mathfrak{l}}, 
\]
is an isomorphism of $\Z$-modules.

The following proposition is a reformulation of the Kazhdan-Lusztig theory for the parabolic category $\mc{O}$ of the Lie algebra $\mathfrak{so}(2m+1)$ (theorems of Brylinski-Kashiwara, Beilinson-Bernstein). 

\begin{prop}
Let ${\bf b} = (0^m)$ (that is $n=0$). The isomorphism $\Psi:  [[\mc{O}^{\mathfrak{l}, \Delta}_{\mathbf{b}}]] \longrightarrow {\mathbb{T}}_{\Z}^{\bf{b}, \mathfrak{l}}$ sends
\[
\Psi([L^{\mathfrak{l}}_{{\bf b}}(\lambda)]) = L^{{\bf b},{\mathfrak{l}}}_{f^{{\bf b}}_{\lambda}}(1), \quad \quad \quad \Psi([T^{\mathfrak{l}}_{{\bf b}}(\lambda)]) 
= T^{{\bf b},{\mathfrak{l}}}_{f^{{\bf b}}_{\lambda}}(1), \quad \quad \text{ for } \lambda \in X_{\mathbf{b}}^{\mathfrak l,+}.
\]
\end{prop}
Note $\widehat{\mathbb{T}}^{\bf{b}, \mathfrak{l}}_{\Z} = {\mathbb{T}}^{\bf{b}, \mathfrak{l}}_{\Z}$ in this case, i.e. no completion is needed.

\begin{proof}
Thanks to \cite[Theorem~5.8]{BW18a}, the $\imath$-canonical basis on $\mathbb{T}^{\bf{b}}$ can be identified with the Kazhdan-Lusztig basis (of type B) on  $\mathbb{T}^{\bf{b}}$. 
Note by \cite[Theorem~5.4]{BW18a} that $\mathbb S^{\mathfrak l} := \sum_{e\neq w \in W_{\mathfrak{l}}} \mathbb{T}^{\bf{b}} \cdot \underline{H_w}$ is a $\Ui$-submodule of  $\mathbb{T}^{\bf{b}}$, and it is actually a based $\Ui$-submodule of  $\mathbb{T}^{\bf{b}}$ with its Kazhdan-Lusztig basis. Therefore the $\imath$-canonical basis  on  $\mathbb{T}^{\bf{b}, \mathfrak{l}}_{\mA}$ in Theorem~\ref{thm:1} can be identified with the basis in the based quotient $\mathbb{T}^{\bf{b}} / \mathbb S^{\mathfrak l}$, which is exactly the parabolic Kazhdan-Lusztig basis. The proposition follows now from the classical Kazhdan-Lusztig theory (cf. \cite{BW18a}). 
\end{proof}

Now we can formulate the super Kazhdan-Lusztig theory for $\mc{O}^{\mathfrak{l}}_{\bf b}$. 
\begin{thm}
 \label{thm:2}
The isomorphism $\Psi:  [[\mc{O}^{\mathfrak{l}, \Delta}_{\mathbf{b}}]] \longrightarrow \widehat{\mathbb{T}}_{\Z}^{\bf{b}, \mathfrak{l}}$ sends
\[
\Psi([L^{\mathfrak{l}}_{{\bf b}}(\lambda)]) = L^{{\bf b},{\mathfrak{l}}}_{f^{{\bf b}}_{\lambda}}(1), \quad \quad \quad \Psi([T^{\mathfrak{l}}_{{\bf b}}(\lambda)]) 
= T^{{\bf b},{\mathfrak{l}}}_{f^{{\bf b}}_{\lambda}}(1), \quad \quad \text{ for } \lambda \in X_{\mathbf{b}}^{\mathfrak l,+}.
\]
\end{thm}

\begin{proof}
Let us briefly explain the idea of the proof from \cite{BW18a}. The crucial new ingredient of this paper (cf. Remark~\ref{rem:Ui} below) is the existence of the $\imath$-canonical basis and dual $\imath$-canonical basis on $\widehat{\mathbb{T}}^{\bf{b}, \mathfrak{l}}$ thanks to Theorem~\ref{thm:1}. Here the dual $\imath$-canonical basis refers to a version of canonical basis where the lattice $\Z[q]$ is replaced by $\Z[q^{-1}]$; see \cite{BW18a}. 

We have already established the version of the theorem for the full category $\mc{O}$ of the Lie superalgebra $\mathfrak{osp}(2m+1 \vert 2n)$ in \cite[Theorem~11.13]{BW18a}. We have the following commutative diagram of $\Z$-modules:
\[
\xymatrix{[[\mc{O}^{\mathfrak{l}, \Delta}_{\mathbf{b}}]]  \ar[r] \ar[d]&\widehat{\mathbb{T}}_{\Z}^{\bf{b}, \mathfrak{l}} \ar[d] \\
[[\mc{O}^{\Delta}_{\mathbf{b}}]] \ar[r]& \widehat{\mathbb{T}}_{\Z}^{\bf{b}}
}
\]
(Note that the vertical arrow on the right is not a based embedding of $\Ui$-modules.) Then the theorem follows from comparison of characters entirely similar to \cite[\S11.2]{BW18a}. Note that this comparison uses only the classical Kazhdan-Lusztig theory. 
\end{proof}

\begin{rem}
  \label{rem:Ui}
In the case of the full category $\mc O$ (i.e., $\mathfrak l$ is the Cartan subalgebra), the theorem goes back to \cite[Theorem 11.13]{BW18a}. Following \cite[Remark 11.16]{BW18a}, the Kazhdan-Lusztig theory for the paraoblic category $\mc{O}^{\mathfrak{l}}_{\bf b}$ with  $\alpha_0 \neq \Pi_{\mathfrak l}$ was a direct consequence of \cite[Theorem 11.13]{BW18a}, via the $\imath$-canonical basis in \cite[Theorem~ 4.25]{BW18a} in the $\UA$-module $ \mathbb{T}^{\bf{b}, \mathfrak{l}}_{\mA}$ in \eqref{eq:T-} with $a_0=0$. 

When $a_0> 0$ (which corresponds to the condition $\alpha_0 \in \Pi_{\mathfrak l}$ on the Levi $\mathfrak l$), the space $ \mathbb{T}^{\bf{b}, \mathfrak{l}}_{\mA}$ in \eqref{eq:T-} is a ${}_\mA\Ui$-module but not a $\UA$-module, and hence Theorem~\ref{thm:1} is needed.
\end{rem}

Denote 
\[
T^{{\bf b}, \mathfrak{l}}_f = M^{{\bf b}, \mathfrak{l}}_f + \sum_{g} 
t^{\bf b}_{gf}(q) M^{{\bf b}, \mathfrak{l}}_g, \qquad \text{ for } t^{\bf b}_{gf}(q) \in \Z[q]. 
\]
By Theorem~\ref{thm:2}, $t^{\bf b}_{gf}(q)$ plays the role of Kazhdan-Lusztig polynomials  for $\mc{O}^{\mathfrak{l}}_{\bf b}$. The following positivity and finiteness results generalize \cite[Theorem 9.11]{BW18a} and follow by the same proof. 

\begin{thm} 
 \label{thm:positivity}
 {\quad}
\begin{enumerate}
\item
We have $t^{{\bf b}, \mathfrak{l}}_{gf}(q) \in \N [q]$.

\item  
The sum $T^{{\bf b}, \mathfrak{l}}_f = M_f^{{\bf b}, \mathfrak{l}} + \sum_{g}t^{\bf b}_{gf}(q)M^{{\bf b}, \mathfrak{l}}_g$ is finite, for all $f$.  
\end{enumerate}
\end{thm}

\begin{rem}
  \label{rem:halfZ}
To formulate a super Kazhdan-Lusztig theory for the parabolic category $\mc{O}$ consisting of modules of half-integer weights for $\mathfrak{osp}(2m+1\vert 2n)$, we use the quantum symmetric pair $(\U, \Ui)$ which is a direct limit of $(\U(\mathfrak{sl}_N), \Ui(\mathfrak{sl}_N) )$ for $N$ odd; cf. \cite[Sections 6, 12]{BW18a}. Theorem~\ref{thm:2} holds again in this setting. 
\end{rem}

\begin{rem}
Following \cite{Bao17}, a simple conceptual modification allows us to formulate a super (type D) Kazhdan-Lusztig theory for the parabolic category $\mc{O}$ consisting of modules of integer (respectively, half-integer) weights for $\mathfrak{osp}(2m\vert 2n)$. To that end, we use the $\imath$-canonical basis of the module \eqref{eq:T-} for the quantum symmetric pair $(\U, \Ui)$, where the parameter is now chosen to be $\kappa=0$ in the notation of \cite{BW18b}. Theorem~\ref{thm:2} holds again in this setting, where the cases new to this paper correspond to the cases $a_0 > 0$. 
\end{rem}


\end{document}